\documentclass{article}
\usepackage{amsthm}
\usepackage{amssymb}
\usepackage{amsmath}
\usepackage{amsfonts}
\usepackage{color}
\usepackage{authblk}

\newtheorem{theorem}{Theorem}
\newtheorem{lemma}{Lemma}
\newtheorem{proposition}{Proposition}
\newtheorem{definition}{Definition}
\newtheorem{example}{Example}
\newtheorem{remark}{Remark}

\title{On Lazard's Valuation\\ and CAD Construction}

\author[*]{Scott McCallum}
\author[**]{Hoon Hong}
\affil[*]{Department of Computing, Macquarie University, NSW 2109, Australia}
\affil[*]{scott.mccallum@mq.edu.au}
\affil[**]{Department of Mathematics, North Carolina State University, Raleigh NC 27695, USA}
\affil[**]{hong@ncsu.edu}

\date{}
\begin{document}

\maketitle

%\begin{frontmatter}

\begin{abstract}
In 1990 Lazard proposed an improved projection operation for cylindrical
algebraic decomposition (CAD). For the proof he introduced a certain notion of
valuation of a multivariate Puiseux series at a point.
However a gap in one of the key supporting
results for the improved projection was subsequently noticed. 
In this report we study a more limited but rigorous concept of Lazard's valuation:
namely, we study Lazard's valuation of a multivariate polynomial at a point.
We prove some basic properties of the limited Lazard valuation
and identify some relationships between valuation-invariance and order-invariance.
\end{abstract}

%\end{frontmatter}

\section{Introduction}

In 1990 Lazard \cite{Lazard:94} proposed an improved projection operation
for cylindrical algebraic decomposition (CAD) which is based on a certain
notion of valuation of a multivariate fractional meromorphic series at a
point. Inherent in \cite{Lazard:94} is the related notion of the
valuation-invariance of an $n$-variate fractional meromorphic series in a
subset of Euclidean $n$-space $\mathbb{R}^n$. Lazard's proposed approach is
in contrast with that of McCallum \cite{McCallum:88, McCallum:98} which is
based on the concept of the order (of vanishing) of a multivariate
polynomial or analytic function at a point, and the related concept of
order-invariance. However a gap in one of the key supporting results of \cite%
{Lazard:94} was subsequently noticed \cite{Collins:98, Brown:01}. 
%Moreover some readers have detected an 
%inherent logical flaw in the kind of valuation
%defined in \cite{Lazard:94}.
This is disappointing because Lazard's proposed approach has some advantages
over other methods.

In \cite{HongMcCallum:14a} we study Lazard's projection.
It is shown there that
Lazard's projection is valid for CAD construction for so-called
well-oriented polynomial sets. The key underlying results relate to
order-invariance rather than valuation-invariance; and the validity proof
builds upon existing results concerning improved projection. While this is
an important step forward we regard it as only a partial validation of
Lazard's approach since the method is not proved to work for non
well-oriented polynomials and it does not involve valuation-invariance.
However it does confirm the soundness of the intuition behind Lazard's
proposed projection. 

In this report we study separately a
more limited but rigorous concept of Lazard's valuation: namely, we study
Lazard's valuation of a multivariate polynomial at a point. This study was
motivated by the hope of remedying more completely the defect of \cite{Lazard:94}.
Section 2 clarifies the
notion and basic properties of the more limited concept of Lazard's
valuation, and identifies some relationships between valuation-invariance and
order-invariance.
Section 3 recalls Lazard's main claim concerning valuation-invariance and CAD. 
His main claim is proved for the special case $n = 3$ under a slightly
stronger hypothesis.
Section 4 discusses some of the
difficulties involved with attempting rigorously to extend the limited
concept of valuation to multivariate fractional meromorphic series.

\section{Definition and basic properties of Lazard's valuation (limited)}

In this section we study Lazard's valuation \cite{Lazard:94}
in a relatively special setting, namely, that of multivariate polynomials
over a field. We shall clarify the notion and basic properties of this
special valuation,
and identify some relationships between valuation-invariance and
order-invariance.

We first define the Lazard valuation in a limited way. This will
allow us to provide simple, straightforward proofs of some basic properties
which are sufficient for some limited uses of this concept.

We recall at the outset the standard algebraic definition of the term
valuation \cite{AtiyahMacDonald, ZariskiSamuel:60}. A mapping $v~:~K -
\{0\}~\rightarrow~\Gamma$ from the multiplicative group of a field $K$ into
a totally ordered abelian group (written additively) $\Gamma$ is said to be
a \emph{valuation} of $K$ if the following two conditions are satisfied:

\begin{enumerate}
\item $v(fg) = v(f) + v(g)$ for all $f$ and $g$;

\item $v(f + g) \ge \min\{v(f), v(g)\}$, for all $f$ and $g$ (with $f+g \neq
0$).
\end{enumerate}

By the same axioms one could define the notion of valuation of a ring \cite%
{Danilov}. (In such a case $\Gamma$ could be a totally ordered abelian
monoid.) Perhaps the simplest and most familiar example of a valuation in
algebraic geometry is the order of an $n$-variate polynomial over a field $k$
at a point $a \in k^n$. That is, the mapping $\mathrm{ord}_a~:~k[x_1,
\ldots, x_n] - \{0\}~\rightarrow~\mathbb{N}$ defined by 
\begin{equation*}
\mathrm{ord}_a(f) = \mbox{the order of}~f~\mbox{at}~a
\end{equation*}
is a valuation of the ring $k[x_1, \ldots, x_n]$. We could extend the
definition of $\mathrm{ord}_a$ to the multiplicative group of the rational
function field $K = k(x_1, \ldots, x_n)$, defining $\mathrm{ord}_a~:~K -
\{0\}~\rightarrow~\mathbb{Z}$ by 
\begin{equation*}
\mathrm{ord}_a(f/g) = \mathrm{ord}_a(f) - \mathrm{ord}_a(g).
\end{equation*}

Let $n \ge 1$. Recall that the \emph{lexicographic order} $\le_l$ on $%
\mathbb{N}^n$ is defined by $v = (v_1, \ldots, v_n) \le_l (w_1, \ldots, w_n)
= w$ if and only if either $v = w$ or there is some $i$, $1 \le i \le n$,
with $v_j = w_j$, for all $j$ in the range $1 \le j < i$, and $v_i < w_i$.
Then $\le_l$ is an \emph{admissible} order on $\mathbb{N}^n$ in the sense of 
\cite{BWK:98}. Indeed $\mathbb{N}^n$, together with componentwise addition
and $\le_l$, forms a totally ordered abelian monoid. The lexicographic order 
$\le_l$ can be defined similarly on $\mathbb{Z}^n$, forming a totally
ordered abelian group.\newline

\begin{definition}
Let $k$ be a field. Let $f(x_1, \ldots, x_n)$ be a nonzero element of the
polynomial ring $k[x_1, \ldots, x_n]$ and let $a = (a_1, \ldots, a_n) \in k^n
$. The \emph{Lazard valuation} $v_a(f)$ of $f$ at $a$ is the element $(v_1,
\ldots, v_n)$ of $\mathbb{N}^n$ least (with respect to $\le_l$) such that
the partial derivative $\partial^{v_1 + \cdots + v_n}f/\partial x_1^{v_1}
\cdots \partial x_n^{v_n}$ does not vanish at $a$.
\end{definition}

We remark that $v_a(f)$ could be defined equivalently to be the element $%
(v_1, \ldots, v_n)$ of $\mathbb{N}^n$ least (with respect to $\le_l$) such
that $f$ expanded about $a$ has a term $c(x_1 - a_1)^{v_1} \cdots (x_n -
a_n)^{v_n}$ with $c \neq 0$. (This is compatible with Lazard's definition in 
\cite{Lazard:94}; see also \cite{HongMcCallum:14a}.) Notice that $v_a(f) = (0,
\ldots, 0)$ if and only if $f(a) \neq 0$. We could extend the definition of $%
v_a$ to the multiplicative group of the rational function field $K = k(x_1,
\ldots, x_n)$, defining $v_a~:~K - \{0\}~\rightarrow~\mathbb{Z}$ by 
\begin{equation*}
v_a(f/g) = v_a(f) - v_a(g).
\end{equation*}
However we shall not need such an extension of the concept for the time
being.

\begin{example}
Let $n = 1$. Then $v_a(f)$ is the familiar order $\mathrm{ord}_a(f)$ of $f$
at $a$. Thus, for instance, if $f(x_1) = x_1^2 - x_1^3$ then $v_0(f) = 2$
and $v_1(f) = 1$.\newline
Let $n = 2$ and $f(x_1, x_2) = x_1 x_2$. Then $v_{(0,0)}(f) = (1,1)$; $%
v_{(1,0)}(f) = (0,1)$; and $v_{(0,1)}(f) = (1,0)$.
\end{example}

Where there is no ambiguity we shall usually omit the qualifier ``Lazard''
from ``Lazard valuation''. We state some basic properties of the valuation $%
v_a(f)$, analogues of properties of the familiar order $\mathrm{ord}_a(f)$.
The first property is the satisfaction of the axioms.

\begin{proposition}
Let $f$ and $g$ be nonzero elements of $k[x_1, \ldots, x_n]$ and let $a \in
k^n$. Then $v_a(fg) = v_a(f) + v_a(g)$ and $v_a(f + g) \ge_l \min\{v_a(f),
v_a(g)\}$ (if $f + g \neq 0$).
\end{proposition}

\begin{proof}
These claims follow since $\mathbb{N}^n$, together with componentwise
addition and $\le_l$, forms a totally ordered abelian monoid.
\end{proof}

\vspace{0.5cm} For the remaining properties we state we shall assume that $k
= \mathbb{R}$ or $\mathbb{C}$.

\begin{proposition}
(Upper semicontinuity of valuation) Let $f$ be a nonzero element of $k[x_1,
\ldots, x_n]$ and let $a \in k^n$. Then there exists a neighbourhood $V
\subset k^n$ of $a$ such that for all $b \in V$ $v_b(f) \le_l v_a(f)$.
\end{proposition}

\begin{proof}
Let $v_a(f) = (v_1, \ldots, v_n)$. By definition of $v_a(f)$, $\partial^{v_1
+ \cdots + v_n}f/\partial x_1^{v_1} \cdots \partial x_n^{v_n}(a) \neq 0$. By
continuity of the function $\partial^{v_1 + \cdots + v_n}f/\partial
x_1^{v_1} \cdots \partial x_n^{v_n}$ at $a$, there exists a neighbourhood $V
\subset k^n$ of $a$ such that for all $b \in V$, $\partial^{v_1 + \cdots +
v_n}f/\partial x_1^{v_1} \cdots \partial x_n^{v_n}(b) \neq 0$. It follows at
once by definition of the valuation that for all $b \in V$ $v_b(f) \le_l
v_a(f)$.
\end{proof}

Let $f \in k[x_1, \ldots, x_n]$. We shall say that $f$ is \emph{%
valuation-invariant} in a subset $S$ of $k^n$ if $v_a(f)$ is constant as $a$
varies in $S$.

\begin{proposition}
Let $f$ and $g$ be nonzero elements of $k[x_1, \ldots, x_n]$ and let $S
\subset k^n$ be connected. Then $fg$ is valuation-invariant in $S$ if and
only if both $f$ and $g$ are valuation-invariant in $S$.
\end{proposition}

\begin{proof}
This proposition is analogous to Lemma A.3 of \cite{McCallum:88}. Its proof,
using Propositions 2.3 and 2.4, is virtually identical to that of the cited
Lemma A.3.
\end{proof}

As noted previously, the above properties are analogues of those of the
familiar order $\mathrm{ord}_a(f)$. The next lemma is in a sense also an
analogue of the familiar order, and is particular to the case $n = 2$.

\begin{lemma}
Let $f(x,y) \in k[x,y]$ be primitive of positive degree in $y$ and
squarefree. Then for all but a finite number of points $(\alpha, \beta) \in
k^2$ on the curve $f(x,y) = 0$ we have $v_{(\alpha,\beta)}(f) = (0,1)$.
\end{lemma}

\begin{proof}
Denote by $R(x)$ the resultant $\mathrm{res}_y(f, f_y)$ of $f$ and $f_y$
with respect to $y$. Then $R(x) \neq 0$ since $f$ is assumed squarefree. Let 
$(\alpha, \beta) \in k^2$, suppose $f(\alpha, \beta) = 0$ and assume that $%
v_{(\alpha,\beta)}(f) \neq (0,1)$. Then $f_y(\alpha,\beta) = 0$. Hence $%
R(\alpha) = 0$. So $\alpha$ belongs to the set of roots of $R(x)$, a finite
set. Now $f(\alpha, \beta) = 0$ and $f(\alpha, y) \neq 0$, since $f$ is
assumed primitive. So $\beta$ belongs to the set of roots of $f(\alpha, y)$,
a finite set.
\end{proof}

Let us consider the relationship between the concepts of order-invariance
and valuation-invariance for a subset $S$ of $k^n$. The concepts are the
same in case $n = 1$ because order and valuation are the same for this case.
For $n = 2$ order-invariance in $S$ does not imply valuation-invariance in $S
$. (For consider the unit circle $S$ about the origin in $k^2$. The order of 
$f(x,y) = x^2 + y^2 - 1$ at every point of $S$ is 1. The valuation of $f$ at
every point $(\alpha, \beta) \in S$ \emph{except} $(\pm 1,0)$ is $(0,1)$.
But $v_{(\pm 1,0)}(f) = (0,2)$.) However for $n = 2$
we can prove the following.

\begin{proposition}
Let $f \in k[x,y]$ be nonzero and $S \subset k^2$ be connected. If $f$ is
valuation-invariant in $S$ then $f$ is order-invariant in $S$.
\end{proposition}

\begin{proof}
Assume that $f$ is valuation-invariant in $S$. Write $f$ as a product of
irreducible elements $f_i$ of $k[x,y]$. By Proposition 2.5 each $f_i$ is
valuation-invariant in $S$. We shall show that each $f_i$ is order-invariant
in $S$. Take an arbitrary factor $f_i$. If the valuation of $f_i$ in $S$ is $%
(0,0)$ then the order of $f_i$ throughout $S$ is $0$, hence $f_i$ is
order-invariant in $S$. So we may assume that the valuation of $f_i$ is
nonzero in $S$, that is, that $S$ is contained in the curve $f_i(x,y) = 0$.
Suppose first that $f_i$ has positive degree in $y$. Now the conclusion is
immediate in case $S$ is a singleton, so assume that $S$ is not a singleton.
Since $S$ is connected, $S$ is an infinite set. By Lemma 2.6 and
valuation-invariance of $f_i$ in $S$, we must have $v_{(\alpha,\beta)}(f_i)
= (0,1)$ for all $(\alpha,\beta) \in S$. Hence $f_i$ is order-invariant in $S
$ (since $\mathrm{ord} f_i = 1$ in $S$). Suppose instead that $f_i = f_i(x)$
has degree 0 in $y$. Since $f_i(x)$ is irreducible it has no multiple roots.
Therefore $v_{(\alpha,\beta)}(f_i) = (1,0)$ for all $(\alpha,\beta) \in S$.
Hence $f_i$ is order-invariant in $S$ (since $\mathrm{ord} f_i = 1$ in $S$).
The proof that $f_i$ is order-invariant in $S$ is finished. We conclude that 
$f$ is order-invariant in $S$ by Lemma A.3 of \cite{McCallum:88}.
\end{proof}

In Section 3 we shall provide an example indicating that Proposition
2.7 is not true for dimension greater than 2.

\section{Valuation-invariance and CAD}

Let $A$ be a set of elements of $\mathbb{Z}[x_1, \ldots, x_n]$. Recall that
an $A$-invariant CAD of $\mathbb{R}^n$ \cite{Collins:75,
Arnon_Collins_McCallum:84} is a partitioning of $\mathbb{R}^n$ into
connected subsets called cells compatible with the zeros of the elements of $%
A$. The output of a CAD algorithm applied to $A$ is a description of an $A$%
-invariant CAD $\mathcal{D}$ of $\mathbb{R}^n$. That is, $\mathcal{D}$ is a
decomposition of $\mathbb{R}^n$ determined by the roots of the elements of $A
$ over the cells of some cylindrical algebraic decomposition $\mathcal{D}%
^{\prime}$ of $\mathbb{R}^{n-1}$; each element of $A$ is sign-invariant
throughout every cell of $\mathcal{D}$.

In this section we first prove a result which implies, roughly speaking,
that many of the cells produced by a CAD algorithm applied to $A$ are
valuation-invariant with respect to each element of $A$.
Next we recall some more terminology, and the main claim, of
\cite{Lazard:94}.
We prove the main claim of \cite{Lazard:94} for $n \le 3$
under a slightly stronger hypothesis.

Recall the fundamental concept of delineability reviewed in
\cite{HongMcCallum:14a}.
Delineability ensures the cylindrical arrangement of the cells in a CAD.
Perhaps the most crucial part of the theory of CADs are theorems providing
sufficient conditions for delineability. Now we can state a result linking
delineability and valuation-invariance.

\begin{proposition}
Let $f \in \mathbb{R}[x,x_n]$ and let $S$ be a connected subset of $\mathbb{R%
}^{n-1}$. Suppose that $f$ is delineable on $S$. Then $f$ is
valuation-invariant in each section (and trivially each sector) of $f$ over $%
S$.
\end{proposition}

\begin{proof}
Let $\theta~:~S \rightarrow \mathbb{R}$ be a real root function of $f$ on $S$
such that $\theta(a)$ has invariant multiplicity $m$ as a root of $f(a, x_n)$%
, as $a$ varies in $S$. ($m$ exists by the second condition of
delineability.) Let $a$ be an arbitrary point of $S$. Then $v_{(a,
\theta(a))} f = (0, \ldots, 0, m)$, since $\partial^m f / \partial x_n^m (a,
\theta(a)) \neq 0$ while $\partial^i f / \partial x_n^i (a, \theta(a)) = 0$
for all $i$ in the range $0 \le i < m$. We've shown that $f$ is
valuation-invariant in the section of $f$ over $S$ which is the graph of $%
\theta$. Observe that wherever $f(a, a_n) \neq 0$ we have $v_{(a, a_n)} f =
(0, \dots, 0)$. Hence $f$ is valuation-invariant in each sector of $f$ over $%
S$.
\end{proof}

An element $f \in \mathbb{Z}[x,x_n]$ is \emph{nullified} by a subset $S$ of $%
\mathbb{R}^{n-1}$ if $f(a, x_n) = 0$ for all points $a \in S$. In case a
cell $S$ is nullified by $f$ the original CAD algorithm does not decompose
the cylinder $S \times \mathbb{R}$ relative to $f$, since in such a case the
whole cylinder $S \times \mathbb{R}$ is sign-invariant with respect to $f$.
In \cite{Lazard:94} Lazard proposed an evaluation process for such an $f$
relative to a sample point $\alpha$ of such an $S$ which he claimed would
ensure that the cylinder $S \times \mathbb{R}$ can be decomposed into
cells which are valuation-invariant with respect to $f$.
This technique is described in slightly more general terms as follows:

\begin{definition}
[Lazard evaluation]
Let $K$ be a field.
(In this section $K = \mathbb{R}$ or, when explicit computation is required,
$K$ is a suitable subfield of $\mathbb{R}$.)
Let $n \ge 2$, $f \in K[x_1, \ldots, x_n]$ nonzero,
and $\alpha = (\alpha_1, \ldots, \alpha_{n-1}) \in K^{n-1}$.
The \emph{Lazard evaluation} $f_{\alpha}(x_n)$ of $f$ at $\alpha$ is 
defined to be the result of the following process
(which determines also nonnegative integers $v_i$, with $1 \le v_i \le n-1$):
\begin{enumerate}
\item[] $f_{\alpha} \leftarrow f$
\item[] For $i \leftarrow 1$ to $n-1$ do
  \begin{enumerate}
  \item[] $v_{i} \leftarrow$ the greatest integer $v$ 
          such that $(x_{i}-\alpha_{i})^{v}~|~f_{\alpha}$
  \item[] $f_{\alpha} \leftarrow f_{\alpha}/(x_{i}-\alpha_{i})^{v_{i}}$
  \item[] $f_{\alpha} \leftarrow f_{\alpha}(\alpha_{i},x_{i+1},\ldots,x_{n})$
  \end{enumerate}
\end{enumerate}
\end{definition}

\begin{remark}
With $K$, $n$, $f$, $\alpha$ and the $v_i$
as in the above definition of Lazard evaluation,
notice that $f(\alpha, x_n) = 0$ (identically) if and only if $v_i > 0$, for some $i$
in the range $1 \le i \le n-1$.
With $\alpha_n \in K$ arbitrary,
notice also that the integers $v_i$, with $1 \le v_i \le n-1$,
are the first $n-1$ coordinates of $v_{(\alpha, \alpha_n)}(f)$.
It will be handy on occasion to refer to the $(n-1)$-tuple
$(v_1, \ldots, v_{n-1})$ as the \emph{Lazard valuation of} $f$ \emph{on} $\alpha$.
\end{remark}

One more definition is needed before we can state Lazard's main claim:

\begin{definition}
[Lazard delineability]
Let $f\in R_n$ be nonzero and
$S$ a subset of $\mathbb{R}^{n-1}$.
We say that $f$ is \emph{Lazard delineable} on $S$ if
\begin{enumerate}
\item the Lazard valuation of $f$ on $\alpha$ is the same for each point
$\alpha \in S$;
\item there exist finitely many continuous functions
$\theta_1 < \cdots < \theta_k$ from $S$ to $\mathbb{R}$, with $k \ge 0$,
such that, for all $\alpha \in S$,
the set of real roots of $f_\alpha (x_n)$ is 
$\{\theta_1 (\alpha), \ldots, \theta_k (\alpha)\}$; and
\item there exist positive integers $m_1, \ldots, m_k$ such that,
for all $\alpha \in S$ and all $i$, $m_i$ is
the multiplicity of $\theta_i(\alpha)$ as a root of $f_{\alpha}(x_n)$.
\end{enumerate} 
We refer to the graphs of the $\theta_i$ as the \emph{Lazard sections} of $f$
over $S$; the regions between successive Lazard sections,
together with the region below the lowest Lazard section
and that above the highest Lazard section, are called \emph{Lazard sectors}.
\end{definition}

\begin{remark}
[Relation between Lazard and ordinary delineability]
Let $f$ and $S$ be as in the above definition of Lazard delineability.
Suppose that $f(\alpha, x_n) \neq 0$ for all $\alpha \in S$.
Then $f$ is Lazard delineable on $S$ if and only if 
$f$ is delineable on $S$ in the usual sense.
\end{remark}

For a finite irreducible basis in $\mathbb{Z}[x_1, \ldots, x_n]$, where $n \ge 2$,
recall that the \emph{Lazard projection} $P_L(A)$ of $A$ is the union
of the set of all leading coefficients of elements of $A$,
the set of all trailing coefficients of elements of $A$,
the set of all discriminants of elements of $A$, and
the set of all resultants of pairs of distinct elements of $A$ 
\cite{HongMcCallum:14a}.
Lazard's main claim, essentially the content of his Proposition 5
and subsequent remarks, could be expressed as follows:

\begin{center}
\medskip
\parbox{11cm}{
Let $A$ be a finite irreducible basis in 
$\mathbb{Z}[x_1, \ldots, x_n]$, where $n \ge 2$.
Let $S$ be a connected subset of $\mathbb{R}^{n-1}$.
Suppose that each element of $P_L(A)$ is valuation-invariant in $S$.
Then each element of $A$ is Lazard delineable on $S$,
the Lazard sections over $S$ of the elements of $A$ are pairwise disjoint,
and each element of $A$ is valuation-invariant in every Lazard section
and sector over $S$ so determined.
}
\medskip
\end{center}

This claim concerns valuation-invariant lifting in relation to $P_L(A)$:
it asserts that the condition, `each element of $P_L(A)$ is valuation-invariant
in $S$', is sufficient for an $A$-valuation-invariant stack
in $\mathbb{R}^n$ to exist over $S$.

\begin{theorem}
Suppose $n \le 3$ and $S$ is a submanifold of $\mathbb{R}^{n-1}$.
Then Lazard's main claim holds.
\end{theorem}

\begin{proof}
Suppose first that $n = 2$.
By remarks in Example 2.2, the hypothesis implies that
each element of $P_L(A)$ is order-invariant in $S$.
Since $n = 2$ and each element of $A$ is irreducible
(hence in particular primitive), no element of $A$ vanishes identically at a point 
of $S$. 
Hence by Theorem 3.1 of \cite{HongMcCallum:14a},
each element of $A$ is delineable on $S$ etc.
By a remark above, Lazard delineability
is equivalent to ordinary delineability for $f$ on $S$ in this case.
The valuation-invariance of each element of $A$ in every section and sector
over $S$ determined by $A$ follows by Proposition 3.1.

Suppose second that $n = 3$.
The conclusions are essentially trivial in case the dimension of $S$ is 0.
So assume henceforth that the dimension of $S$ is positive.
By the hypothesis and Proposition 2.7, each element of $P_L(A)$ 
is order-invariant in $S$.
Hence by Theorem 3.1 of \cite{HongMcCallum:14a},
each element of $A$ either vanishes identically on $S$ or is delineable on $S$.
Since the dimension of $S$ is positive and each element of $A$ is irreducible
(hence in particular primitive), no element of $A$ vanishes identically
on $S$ (Lemma A.2 of \cite{McCallum:88}).
Hence each element of $A$ is delineable on $S$ and (again by Theorem 3.1
of \emph{op. cit.}) the
sections over $S$ of the elements of $A$ are pairwise disjoint.
By a remark above, Lazard delineability
is equivalent to ordinary delineability for $f$ on $S$ in this case.
The valuation-invariance of each element of $A$ in every section and sector
over $S$ determined by $A$ follows by Proposition 3.1.
\end{proof}

Now we present an example showing that valuation-invariance does not imply
order-invariance when $n > 2$. Let $f(x,y,z) = x z - y^2$ and let $S$ be the 
$z$-axis in $\mathbb{R}^3$. Now $f$ is valuation-invariant in $S$, since the
valuation of $f$ at each point of $S$ is $(0,2,0)$. But $f$ is not
order-invariant in $S$, since $\mathrm{ord}_{(0,0,0)} f = 2$ and $\mathrm{ord%
}_{(0,0,\alpha)} f = 1$ for $\alpha \neq 0$. To our minds this example casts
some doubt on the truth of Lazard's main claim for $n > 3$. However we do not
have a direct counter-example for Lazard's main claim.

\section{Extension of Lazard's valuation: some difficulties}

Let $k$ be $\mathbb{R}$ or $\mathbb{C}$. In this subsection we discuss some
of the problems with trying to extend the definitions and basic theory of
valuations of $k[x_1, \ldots, x_n]$, as outlined in Section 2, to
larger rings and fields. We first observe that both the order and Lazard
valuation could just as easily be defined in the same manner for nonzero
elements of the ring of analytic functions defined in some open set $U
\subset k^n$. The basic theory (Propositions 2.3 to 2.5) carries over with
very little change.

Now let $a = (a_1, \ldots, a_n) \in k^n$. We consider the formal power
series ring 
\begin{equation*}
R = k[[x - a]] = k[[x_1 - a_1, \ldots, x_n - a_n]].
\end{equation*}
We could define the order $\mathrm{ord}_a(f)$ of $f \in R$ at $a$ in the
usual way. However -- unless $f$ is assumed to be convergent in some
neighbourhood of $a$ (i.e. analytic near $a$, discussed above) -- it is not
in general possible to expand $f$ about $b \neq a$ but near $a$. In
particular analogues of Propositions 2.4 and 2.5 have no meaning in general
in this case. Similar remarks apply to the Lazard valuation in this case.
This elementary consideration points to the need to carefully incorporate
some notion of convergence if one wishes to rectify the flawed theory
relating to Lazard's valuation for multivariate Puiseux series identified in
Section 2 of \cite{HongMcCallum:14a}. 
Yet the convergence issue for multivariate Puiseux series is
thorny -- a given such series in $x - a$ may have a slender region of
convergence (neither a full nor punctured neighbourhood of $a$).

Next we consider extension of the valuations $\mathrm{ord}_a$ and $v_a$ to
the field of rational functions $K = k(x_1, \ldots, x_n)$. As mentioned in
Section 2, $\mathrm{ord}_a$ could be defined by the equation $\mathrm{%
ord}_a(f/g) = \mathrm{ord}_a(f) - \mathrm{ord}_a(g)$, for $f$ and $g$
nonzero elements of $K$. Proposition 2.3 remains valid for this extension
but Proposition 2.4 does not (as we could take $n = 1$, $a_1 = 0$, $f = 1$
and $g = x_1$). However we could formulate the following analogue of
Proposition 2.4:

\begin{proposition}
Let $f/g$ be a nonzero element of $K$, let $U \subset k^n$ be an open set
throughout which $g \neq 0$, and let $a \in U$. Then there exists a
neighbourhood $V \subset U$ of $a$ such that for all $b \in V$ $\mathrm{ord}%
_b(f/g) \le \mathrm{ord}_a(f/g)$.
\end{proposition}

A direct proof (adapting that of Proposition 2.4) could easily be given. In
fact this is a special case of the analogue of Proposition 2.4 for the case
of analytic functions mentioned above.

Similar remarks and an analogue of the above proposition apply to the Lazard
valuation.

Now in case $k = \mathbb{C}$ there is an interesting relationship between
the rational function field $K$ and the field of all $n$-variate Puiseux
series $\mathbb{C}_n(x;a)$, defined in Section 2 of
\cite{HongMcCallum:14a}. Indeed there is a
natural embedding of $K$ into $\mathbb{C}_n(x;a)$, for every point $a$. So,
with a view toward rectifying the flawed theory relating to Lazard's
valuation for multivariate Puiseux series \cite{HongMcCallum:14a}, one might wish
to find a way to relax the hypotheses of the above proposition (for $v_a$)
without invalidating it. However it is by no means clear how to do this.

Finally we consider possible extension of the valuations and the associated
basic theory to fractional power series rings. Both the order and Lazard
valuation could be defined for an Abhyankar-Jung power series ring 
\cite{HongMcCallum:14a} such as $\mathbb{C}\{x_1, \ldots, x_{n-1},
x_n^{1/q}\}$. Moreover, it is conceivable that a suitable analogue of
Proposition 2.4 could be proved for a case like this. However further
extension -- to some suitable subring of $\mathbb{C}(x;a)$ which includes
all the desired roots, and in which one has a reasonable notion of
convergence -- would seem challenging in view of the difficulties noted
above.

\section{Conclusion, further work}

Mindful of the potential benefits of Lazard's approach to projection in \cite%
{Lazard:94}, yet conscious of the flawed justification provided, we embarked
on investigation of both the Lazard projection and valuation. 
In \cite{HongMcCallum:14a} we found that
Lazard's projection is valid for CAD construction for well-oriented
polynomial sets. The validity proof uses order-invariance instead of
valuation-invariance, and builds upon existing results \cite{McCallum:88,
McCallum:98, Brown:01} about improved projection. 
In this report we proved some basic
properties of Lazard's valuation in the special setting of real or complex
multivariate polynomials, and identified some relationships between
valuation-invariance and order-invariance.

Further work could usefully be done in a number of directions. It would be
desirable to extend the CAD algorithm, with improved projection, to apply to
non-well-oriented sets. It would be nice to have a more streamlined,
condensed account of the theory of improved projection for CAD, which is
currently scattered across several journal articles spanning nearly three
decades. It would be interesting to try to pursue further the notion and
theory of the Lazard valuation, which effort could yield some worthwhile
algorithmic and theoretical improvements. 
In particular, it would be worthwhile to try to prove Lazard's main claim
(with no restriction to $n \le 3$).
Examination of the other ideas
suggested in \cite{Lazard:94} could also be fruitful.

\section*{Acknowledgements}

We would like to acknowledge a grant from the
Korea Institute of Advanced Study (KIAS) which supported a visit to KIAS by
both authors in 2009. This visit made possible the start of our
collaboration on this work. Scott McCallum also acknowledges his University
for providing an Outside Studies Program External Fellowship in 2012.
 Hoon Hong acknowledges the partial support from the grant US NSF 1319632.

\end{document}